\documentclass[reqno,a4paper,11pt]{article}

\usepackage{amsmath,amssymb,amsthm}

\usepackage{mathrsfs}

\usepackage{mathtools}
\usepackage{commath}

\usepackage{thmtools}
\usepackage{thm-restate}

\usepackage{cases}
\usepackage{enumitem}
\setlist[enumerate,1]{label=(\roman*)}

\usepackage[pdftex, pdfborderstyle={/S/U/W 0}]{hyperref}
\hypersetup{
    colorlinks=true,
    linkcolor=red,
    citecolor=blue,
}

\usepackage{cleveref}

\usepackage{etoolbox}

\usepackage{comment}

\usepackage[numbers]{natbib}

\numberwithin{equation}{section}

\ifdefined\thmcolor
\declaretheoremstyle[
  shaded={bgcolor=\thmcolor}
]{plain}
\else
\fi

\ifdefined\defcolor
\declaretheoremstyle[
  headfont=\normalfont\bfseries,
  bodyfont=\normalfont,
  shaded={bgcolor=\defcolor}
]{noital}
\else
\declaretheoremstyle[
  headfont=\normalfont\bfseries,
  bodyfont=\normalfont,
]{noital}
\fi

\declaretheorem[style=plain,numberwithin=section,name=Theorem]{theorem}

\declaretheorem[style=plain,sibling=theorem,name=Lemma]{lemma}
\declaretheorem[style=plain,sibling=theorem,name=Corollary]{corollary}
\declaretheorem[style=plain,sibling=theorem,name=Conjecture]{conjecture}
\declaretheorem[style=plain,sibling=theorem,name=Claim]{claim}

\declaretheorem[style=plain,sibling=theorem,name=Question]{question}
\declaretheorem[style=plain,sibling=theorem,name=Observation]{observation}

\declaretheorem[style=plain,numbered=no,name=Theorem]{theorem-n}
\declaretheorem[style=plain,numbered=no,name=Proposition]{proposition-n}
\declaretheorem[style=plain,numbered=no,name=Lemma]{lemma-n}
\declaretheorem[style=plain,numbered=no,name=Corollary]{corollary-n}
\declaretheorem[style=plain,numbered=no,name=Conjecture]{conjecture-n}
\declaretheorem[style=plain,numbered=no,name=Claim]{claim-n}
\declaretheorem[style=plain,numbered=no,name=Fact]{fact-n}
\declaretheorem[style=plain,numbered=no,name=Open Problem]{openproblem-n}
\declaretheorem[style=plain,numbered=no,name=Question]{question-n}
\declaretheorem[style=plain,numbered=no,name=Observation]{observation-n}

\declaretheorem[style=noital,numbered=no,name=Remark]{remark-n}
\declaretheorem[style=noital,numbered=no,name=Definition]{definition-n}
\declaretheorem[style=noital,numbered=no,name=Construction]{construction-n}
\declaretheorem[style=noital,numbered=no,name=Example]{example-n}

\newcommand{\defined}{\mathrel{\coloneqq}}

\DeclarePairedDelimiter{\p}{\lparen}{\rparen}

\let\b\relax
\DeclarePairedDelimiter{\b}{\lbrack}{\rbrack}

\newcommand{\st}{\mathbin{\colon}}

\undef{\set}
\DeclarePairedDelimiter{\set}{\lbrace}{\rbrace}

\undef{\emptyset}
\newcommand{\emptyset}{\varnothing}

\newcommand{\from}{\colon}

\DeclarePairedDelimiter{\floor}{\lfloor}{\rfloor}
\DeclarePairedDelimiter{\ceil}{\lceil}{\rceil}

\let\d\relax
\newcommand{\d}{\, \mathrm{d}}

\undef{\mod}
\newcommand{\mod}[1]{\ (\mathrm{mod}\ #1)}

\undef{\abs}
\DeclarePairedDelimiterX{\abs}[1]
  {\lvert}{\rvert}{\ifblank{#1}{\,\cdot\,}{#1}}

\undef{\norm}
\DeclarePairedDelimiterX{\norm}[1]
  {\lVert}{\rVert}{\ifblank{#1}{\,\cdot\,}{#1}}

\DeclarePairedDelimiterX{\inner}[2]
  {\langle}{\rangle}{\ifblank{#1}{\,\cdot\,}{#1},\ifblank{#2}{\,\cdot\,}{#2}}

\DeclarePairedDelimiterX{\absinner}[2]
  {|\langle}{\rangle|}{\ifblank{#1}{\,\cdot\,}{#1},\ifblank{#2}{\,\cdot\,}{#2}}

\DeclareMathDelimiter{\given}
  {\mathbin}{symbols}{"6A}{largesymbols}{"0C}

\DeclareMathOperator{\Prob}{\mathbb{P}}
\DeclarePairedDelimiterXPP{\prob}[1]
  {\Prob}{\lparen}{\rparen}{}
  {\renewcommand{\given}{\nonscript\;\delimsize\vert\nonscript\;\mathopen{}}#1}

\DeclareMathOperator{\Expec}{\mathbb{E}}
\DeclarePairedDelimiterXPP{\expec}[1]
  {\Expec}{\lparen}{\rparen}{}
  {\renewcommand{\given}{\nonscript\;\delimsize\vert\nonscript\;\mathopen{}}#1}

\DeclareMathOperator{\Var}{Var}
\DeclarePairedDelimiterXPP{\var}[1]
  {\Var}{\lparen}{\rparen}{}
  {\renewcommand{\given}{\nonscript\;\delimsize\vert\nonscript\;\mathopen{}}#1}

\DeclareMathOperator{\Cov}{Cov}
\DeclarePairedDelimiterXPP{\cov}[2]
  {\Cov}{\lparen}{\rparen}{}{#1,#2}

\newcommand{\sseq}{\subseteq}

\newcommand{\EE}{\mathbb{E}}

\usepackage{geometry}
\usepackage{scrextend}

\usepackage[T1]{fontenc}
\usepackage{titlesec}

\usepackage{tikz}

\titleformat{\section}{\centering\bfseries\scshape\Large}{\thesection}{1em}{}
\titleformat{\subsection}{\bfseries\scshape\large}{\thesubsection}{1em}{}

\newcommand{\red}{\text{Red}}
\newcommand{\blue}{\text{Blue}}
\newcommand{\nbhd}[2]{N^{(#1)}(#2)}
\newcommand{\colours}{\set{\red,\blue}}

\begin{document}

\title{\textsc{\bfseries Hypercube geodesics with few colour changes}}

\renewcommand{\thefootnote}{\fnsymbol{footnote}}

\author{\textsc{Lawrence Hollom}\footnotemark[1]}

\footnotetext[1]{\href{mailto:lawrence.hollom@epfl.ch}{lawrence.hollom@epfl.ch}. Institute of Mathematics, EPFL, Lausanne, Switzerland}

\renewcommand{\thefootnote}{\arabic{footnote}}

\date{}

\maketitle

\begin{abstract}
	What is the maximum, over all 2-colourings of the edges of the $n$-dimensional hypercube $Q_n$, of the minimal number of times a path between a vertex $v$ and its antipode $\bar{v}$ changes colour?
	A conjecture of Norine, in a form due to Feder and Subi, states that this maximum should be 1.
	The previous best-known upper bound on the number of colour changes was $(\tfrac{5}{16} + o(1))n$ due to Kirchweger, Peitl, Subercaseaux, and Szeider.
	We improve this bound and answer a question of Leader and Long by finding a geodesic path with at most $(\tfrac{\pi}{2} + o(1))\sqrt{n}$ colour changes. 
	In fact, we show that this is the expected number of colour changes for a uniformly random start vertex.
	This is optimal (up to the constant) when the start vertex is chosen uniformly at random.
\end{abstract}


\section{Introduction}
\label{sec:intro}

The hypercube is a central object in combinatorics, and many problems concerning the structures that can be found on taking subsets or colourings of the edges have been studied. 
While a result of Alon, Radoi{\v{c}}i{\'c}, Sudakov, and Vondr{\'a}k \cite{ARSV06} precisely characterises those graphs for which a monochromatic copy will appear in a sufficiently large $k$-coloured hypercube $Q_n$ (for $k$ fixed), much remains unknown about structures which grow with the dimension $n$, or if one only wants approximately monochromatic structures.
One example of such a problem is a conjecture of Norine (see \cite{Nor08}), who asked whether every antipodal 2-colouring of the hypercube $Q_n$ (a colouring where antipodal edges receive different colours) contains a monochromatic path between some pair of antipodal vertices.
This conjecture was modified by Feder and Subi \cite{FS13} to the following.

\begin{conjecture}
	\label{conj:path}
	Every 2-colouring of $E(Q_n)$ contains a path between some pair of antipodal vertices which changes colour at most once.
\end{conjecture}

It was proved in \cite{FS13} that one can find a monochromatic path between some pair of vertices at Hamming distance at least $\ceil{n/2}$, and thus there is a pair of antipodal vertices and a path between them with at most $\floor{n/2}$ colour changes.
Long \cite{Lon13} proved that subgraphs of $Q_n$ of large average degree must contain exponentially long paths, and later Leader and Long \cite{LL14} showed that one can also find geodesics of length $\ceil{n/2}$ in such subgraphs, proving the existence of a geodesic between antipodes with at most $\floor{n/2}$ colour changes.
They moreover conjectured that the path in \Cref{conj:path} can be taken to be a geodesic and posed the following, weaker, question.

\begin{question}
	\label{q:few}
	Is it true that for every 2-colouring of $E(Q_n)$, there exist two antipodal vertices $x$ and $x'$ that are joined by a path that changes colour $o(n)$ times?
\end{question}

Our main result is the following, which provides an affirmative answer to \Cref{q:few}.

\begin{theorem}
	\label{thm:main}
	For any 2-colouring of the edges of the hypercube $Q_n$, there is a vertex $v\in V(Q_n)$ and a geodesic path from $v$ to its antipode $\bar{v}$ which changes colour at most $(\tfrac{\pi}{2} + o(1))\sqrt{n}$ times.
\end{theorem}

Indeed, this continues the pattern of geodesic paths being sufficient.
We note that \Cref{thm:main} constitutes the first sub-linear bound on this problem.
Before our result, Dvo\v{r}\'{a}k \cite{Dvo19} used a careful consideration colourings of copies of $Q_3 \sseq Q_n$, proved that there always exists a pair of antipodal vertices and a geodesic between them with at most $(\tfrac{3}{8}+o(1))n$ colour changes.
In a different direction, \Cref{conj:path} is also known to hold for values of $n$ up to 8 \cite{WW18,ZBHKC17,FS24,KPSS25} by application of computer search with SAT-solvers.

Combining these two approaches, Kirchweger, Peitl, Subercaseaux, and Szeider \cite{KPSS25} found a more efficient encoding of the problem, allowing them to not only prove the $n=8$ case of the conjecture, as mentioned above, but to also use an automated analysis of $Q_6$ to improve Dvo\v{r}\'{a}k's bound to $(\tfrac{5}{16} + o(1))n$ colour changes.

\Cref{q:few} was repeated by Ivan and Johnson as part of a larger collection of open problems \cite{Imre23}.
They therein raised the question of what can be said if we insist that the vertex $v$ at the start of the geodesic is chosen uniformly at random in the hypercube, as follows.

\begin{question}
	\label{q:random}
	How large can be the minimum number of colour-changes on a path between a randomly chosen pair of antipodal vertices? In other words, determine the maximum (over all edge colourings of $Q_n$) of the quantity
	\begin{equation}
		\label{eq:q-random}
		\frac{1}{2^n} \sum_{v\in Q_n} (\text{the minimum number of colour changes on a }v\text{-to-}\bar{v}\text{ path})
	\end{equation}
\end{question}

Our proof chooses the end of the geodesic uniformly at random and -- as will be discussed in the concluding remarks -- there is a colouring of $Q_n$ where the value in \eqref{eq:q-random} is $\Omega(\sqrt{n})$.
This provides a strong resolution of \Cref{q:random} for geodesics as well as more general paths.



\subsection{Intuition and proof outline}

We prove that there is a choice of vertex $v$ and a geodesic path from $v$ to $\bar{v}$ with few colour changes.
Our starting vertex $v$ is chosen uniformly at random in $V(Q_n)$, but even so we must be careful in our choice of path: if we attempt to pick a geodesic from $v$ to $\bar{v}$ greedily, there is no immediate way of ruling out the possibility of running into lots of ``dead-end'' vertices -- vertices $x$ where our geodesic path arrives via, say, a red edge, but all edges from $x$ away from $v$ are blue, and so we are forced to change colour. 
A key intuition for our proof is the following informal heuristic, which says that arriving at such a dead end should be very unlikely.

Fix a vertex $x$, and consider picking a vertex $v$ uniformly at random amongst all vertices at a distance $k$ from $x$, and assume that we want our geodesic to start from $v$ and go through $x$. 
The only way in which $x$ could be a dead end is if either all of its red neighbours (that is, neighbouring vertices in $Q_n$ connected by a red edge) or all of its blue neighbours are closer to $v$ than $x$ is.
However, if $x$ is in many red edges and many blue edges and $k$ is not too close to $n$, then this is very unlikely to happen.
If on the other hand $x$ has very few edges of one colour, say, blue, then while it does have a non-negligible chance of being a dead end in blue, most paths arriving at $x$ use a red edge, from which there are plenty of options for continuing the geodesic path with another red edge.

To formalise this intuition, we consider the function $f_c(x,y)$, which is the minimum number of colour changes in a geodesic between vertices $x$ and $y$ ending in colour $c$.
We in fact compute a weighted average of this function around the vertex $y$, weighted by the proportion of the edges around $y$ in each colour.
In fact, we compute several version of such an average, and relate them to each other in sequence (see \Cref{fig:fig} for a pictorial representation of these averages).

The crux of the argument is to step from weighting by the edges from $y$ towards $x$ to weighting by the edges away from $x$. But, if $y$ is fixed and $x$ is chosen uniformly at random at a specific distance from $y$, then the number of red edges from $y$ towards $x$ is distributed as a hypergeometric random variable, which is very easy to control.
Linearity of expectation then allows us to sum these changes over a random path from $x$ to $\bar{x}$, and the result follows.


\section{Proof of the main theorem}
\label{sec:proof}

We first give a complete description of all the notation that we will use throughout the proof, and then use this notation to state our key lemma, \Cref{lem:key}, in \Cref{subsec:notation} before proceeding to give a proof of this lemma in \Cref{subsec:proof}.


\subsection{Notation and reduction to a key lemma}
\label{subsec:notation}

Throughout this section we work on the graph $Q_n$ of the hypercube in $n$ dimensions with a colouring $\chi\from E(Q_n)\to\colours$.
The value of $n$ will be fixed throughout.
Denote the vertex set and edge set of $Q_n$ as $V$ and $E$ respectively.
We write $x\sim y$ to mean that vertices $x$ and $y$ are adjacent, and $x\sim_c y$ to mean that they are connected by an edge with colour $c$.
We call the two colours $\red$ and $\blue$.

We write $d(x,y)$ for the Hamming distance from $x$ to $y$ (i.e.\ the graph distance in $Q_n$), so for example $x\sim y$ implies $d(x,y) = 1$, and we write
\begin{equation*}
	\nbhd{k}{x} \defined \set{y \in V \st d(x, y) = k}
\end{equation*}
for the $k$th neighbourhood of $x$.

We now introduce our non-standard notation.
For all $x,y\in V$ and $c\in\colours$, we define the sets $I$ and $J$, which can be thought as the edges from a vertex respectively towards and away from the start of the path that we construct.
\begin{align*}
	I_c^x(y) &\defined \set{z\in V\st z\sim_c y, \; d(x,z) = d(x, y) - 1}, \\
	I^x(y) &\defined \set{z\in V\st z\sim y, \; d(x,z) = d(x, y) - 1}, \\
	J_c^x(y) &\defined \set{z\in V\st z\sim_c y, \; d(x,z) = d(x, y) + 1}, \;\text{ and}\\
	J^x(y) &\defined \set{z\in V\st z\sim y, \; d(x,z) = d(x, y) + 1}.
\end{align*}
We then define the following number, which will be explained in \Cref{subsec:proof}.
\begin{equation}
	\label{eq:s-def}
	S^x(y) \defined \abs[\bigg]{\frac{\abs{J^x_\red(y)}}{n - d(x, y)} - \frac{\abs{I^x_\red(y)}}{d(x,y)}} = \abs[\bigg]{\frac{\abs{J^x_\blue(y)}}{n - d(x, y)} - \frac{\abs{I^x_\blue(y)}}{d(x,y)}}.
\end{equation}
For now we just note that the above equality comes from the fact that
\begin{equation*}
	\abs{J_\red^x(y)} + \abs{J_\blue^x(y)} = n - d(x,y) \quad \text{and} \quad \abs{I_\red^x(y)} + \abs{I_\blue^x(y)} = d(x,y).
\end{equation*}

We now discuss our use of random variables.
We select two points $u,v\in V$ independently uniformly at random. 
$\EE$ will be the unconditioned expectation, and for $k\in[n]$ (where $[n]=\set{1,2,\dotsc,n}$), we will use the shorthand
\begin{equation*}
	\EE_k[\;\cdot\;] \defined \EE[\;\cdot\; \given d(u,v) = k].
\end{equation*}

Write $f_c(x, y)$ for the minimum number of colour changes on a geodesic path from $x$ to $y$ ending in colour $c$.
We allow the path to change colour at $y$, even if no edge of that colour appears in the path; for example, if $x\sim_\blue y$, then $f_\blue(x,y) = 0$ and $f_\red(x, y) = 1$.
We record a simple observation about these values.

\begin{observation}
	\label{obs:close}
	For all vertices $x, y \in V$, we have
	\begin{equation*}
		\abs[\big]{f_\red(x,y) - f_\blue(x,y)} \leq 1.
	\end{equation*}
\end{observation}
This follows from our ability to swap between colours at $y$ if we wish.

We now define three weighted averages, which will be the central objects of our proof.
\begin{align}
	\label{eq:def-f} f(x, y) &\defined \frac{1}{n-d(x, y)} \sum_{z\in J^x(y)} f_{\chi(yz)}(x, y) = \frac{1}{n-d(x, y)} \sum_{c\in\colours} \abs{J_c^x(y)} \, f_c(x,y), \\
	\label{eq:def-g} g(x, y) &\defined \frac{1}{d(x, y)} \sum_{z\in I^x(y)} f_{\chi(yz)}(x, z), \quad\text{ and} \\
	\label{eq:def-h} h(x, y) &\defined \frac{1}{d(x, y)} \sum_{z\in I^x(y)} f_{\chi(yz)}(x, y) = \frac{1}{d(x,y)} \sum_{c\in\colours} \abs{I_c^x(y)} f_c(x,y).
\end{align}

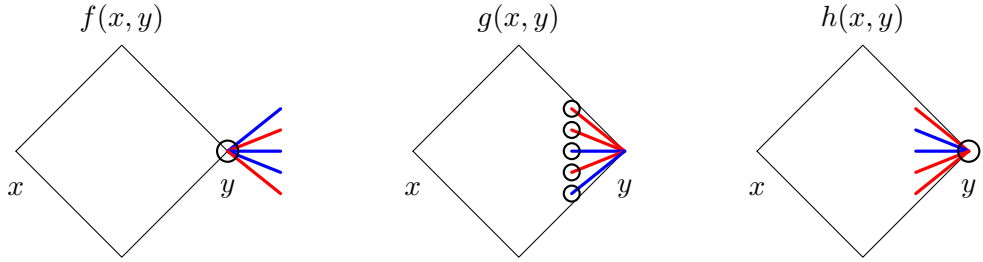
\begin{figure}[ht]
	\centering
	\begin{tikzpicture}[scale=0.7, line cap=round, line join=round]
		\begin{scope}[shift={(0,0)}]
			\node at (0,2.5) {$f(x,y)$};

			\draw (-2,0) -- (0, 2) -- (2, 0) -- (0, -2) -- (-2, 0);

			\node at (-2,-0.7) {$x$};
			\node at (2,-0.7) {$y$};

			\draw[very thick, blue] (2,0) -- (3,0.8);
			\draw[very thick, red] (2,0) -- (3,0.4);
			\draw[very thick, blue] (2,0) -- (3,0.0);
			\draw[very thick, blue] (2,0) -- (3,-0.4);
			\draw[very thick, red] (2,0) -- (3,-0.8);

			\draw[thick] (2,0) circle (0.2);
		\end{scope}
		
		\begin{scope}[shift={(7.5,0)}]
			\node at (0,2.5) {$g(x,y)$};

			\draw (-2,0) -- (0, 2) -- (2, 0) -- (0, -2) -- (-2, 0);

			\node at (-2,-0.7) {$x$};
			\node at (2,-0.7) {$y$};

			\draw[very thick, red] (2,0) -- (1,0.8);
			\draw[very thick, red] (2,0) -- (1,0.4);
			\draw[very thick, blue] (2,0) -- (1,0.0);
			\draw[very thick, red] (2,0) -- (1,-0.4);
			\draw[very thick, blue] (2,0) -- (1,-0.8);

			\draw[thick] (1,0) circle (0.15);
			\draw[thick] (1,0.4) circle (0.15);
			\draw[thick] (1,0.8) circle (0.15);
			\draw[thick] (1,-0.4) circle (0.15);
			\draw[thick] (1,-0.8) circle (0.15);
		\end{scope}
		
		\begin{scope}[shift={(14,0)}]
			\node at (0,2.5) {$h(x,y)$};

			\draw (-2,0) -- (0, 2) -- (2, 0) -- (0, -2) -- (-2, 0);

			\node at (-2,-0.7) {$x$};
			\node at (2,-0.7) {$y$};

			\draw[very thick, red] (2,0) -- (1,0.8);
			\draw[very thick, blue] (2,0) -- (1,0.4);
			\draw[very thick, blue] (2,0) -- (1,0.0);
			\draw[very thick, red] (2,0) -- (1,-0.4);
			\draw[very thick, red] (2,0) -- (1,-0.8);

			\draw[thick] (2,0) circle (0.2);
		\end{scope}
	\end{tikzpicture}
	\caption{A pictorial representation of the functions $f$, $g$, and $h$. Each function averages $f_c(x,z)$ (the minimum number of colour changes on a geodesic path from $x$ to $z$ ending in colour $c$) over the edges shown in colour, where $c$ is the colour of the edge and $z$ is the circled vertex. The large diamond represents the sub-cube between $x$ and $y$.}
	\label{fig:fig}
\end{figure}

We are now ready to state our key lemma, which is as follows.

\begin{lemma}
	\label{lem:key}
	Fix an integer $k\in [n]$.We have
	\begin{equation}
		\label{eq:key}
		\EE_k \b[\big]{f(u,v)} \leq \frac{1}{2} (1+o(1)) \sum_{j=1}^k \sqrt{\frac{j(n-j)}{n}}\p[\Big]{\frac{1}{n-j} + \frac{1}{j}}.
	\end{equation}
\end{lemma}

Once we have \Cref{lem:key}, it is easy to deduce \Cref{thm:main}:

\begin{proof}[Proof of \Cref{thm:main} assuming \Cref{lem:key}]
	We compute an upper bound on $\EE_n[f(u,v)]$, which is the expected optimal number of colour changes between $u$ and $v$ which are conditioned to be antipodal.
	\begin{equation*}
		\frac{1}{2}\sum_{j=1}^{n-1} \sqrt{\frac{j(n-j)}{n}}\p[\Big]{\frac{1}{n-j} + \frac{1}{j}} 
		= \sum_{j=1}^{n-1}\sqrt{\frac{n-j}{nj}} \leq \int_0^{n-1}\sqrt{\frac{n-x}{nx}}\d x
		\leq\sqrt{n} \int_0^{1}\sqrt{\frac{1-y}{y}}\d y = \frac{\pi}{2}\sqrt{n}.
	\end{equation*}
	We have used in the above that the function $\sqrt{(n-x)/(nx)}$ is decreasing in $x$ for $x > 0$. 
	There is thus a choice of $v$ for which there is a path from $v$ to $\bar{v}$ at most the above number of colour changes, from which \Cref{thm:main} follows immediately.
\end{proof}



\subsection{Proof of the key lemma}
\label{subsec:proof}

We prove \Cref{lem:key} by induction on $k$. 
Indeed, the base case $k=0$ is trivial as $f(v,v) = 0$ and the sum in \eqref{eq:key} is over no terms, so is also zero, so it remains to prove the inductive step.
The following three simple claims form the engine of our proof: we step from $\EE_{k-1}\b{f(u,v)}$ to $\EE_k\b{g(u,v)}$, then to $\EE_k\b{h(u,v)}$, and finally to $\EE_k\b{f(u,v)}$, completing the induction.

\begin{claim}
	\label{claim:step1}
	Fix an integer $k\in [n]$. We then have $\EE_{k-1}\b{f(u,v)} = \EE_k \b{g(u,v)}$.
\end{claim}

\begin{proof}
	This claim in fact holds even if we condition on the value of $u$.
	Indeed, let $E_k(u)$ be the set of edges from $\nbhd{k-1}{u}$ and $\nbhd{k}{u}$.
	Both of the expectations in \Cref{claim:step1} are equal to
	\begin{equation*}
		\frac{1}{\abs{E_k(u)}} \sum_{e\in E_k(u)} f_{\chi(e)}(u, v_e),
	\end{equation*}
	where $v_e$ is the vertex of $e$ in $\nbhd{k-1}{u}$.
	Indeed, this follows immediately from the observation that we may count $E_k(u)$ by selecting a vertex $x\in \nbhd{k-1}{u}$ and a neighbour $y$ of $x$ in $\nbhd{k}{u}$, or by first selecting $y$ in $\nbhd{k}{u}$, and then selecting a neighbour $x$ of $y$ in $\nbhd{k-1}{u}$.
\end{proof}

\begin{claim}
	\label{claim:step2}
	For all $x,y \in V$, we have $h(x,y) \leq g(x,y)$.
\end{claim}

\begin{proof}
	Observing the definitions \eqref{eq:def-g} and \eqref{eq:def-h} of $g$ and $h$ respectively, it is clear that it suffices to prove that, for a vertex $x\in V$ and an edge $yz\in E$ with $y$ closer to $x$ than $z$ is, we have
	\begin{equation*}
		f_{\chi(yz)}(x, y) \geq f_{\chi(yz)}(x,z).
	\end{equation*}
	Indeed, this holds because a geodesic from $x$ to $y$ ending in colour $yz$ may be extended along the edge $yz$ without incurring an additional colour change.
\end{proof}

\begin{claim}
	\label{claim:step3}
	Fix an integer $k\in [n]$. We then have $\EE_k\b{f(u,v)} \leq \EE_k \b{h(u,v)} + \EE_k\b{S^u(v)}$.
\end{claim}

\begin{proof}
	We may unwrap the definitions \eqref{eq:s-def}, \eqref{eq:def-f}, and \eqref{eq:def-h} to compute that
	\begin{align*}
		\EE_k\b{f(u,v)} - \EE_k\b{h(u,v)} &= \sum_{c\in\colours}\p[\bigg]{\frac{\abs{J_c^u(v)}}{n-k} - \frac{\abs{I_c^u(v)}}{k}} \, \EE_k\b[\big]{f_c(u,v)} \\
		&\leq \EE_k \b[\Big]{\,S^u(v) \cdot \abs[\big]{f_\red(u,v) - f_\blue(u,v)}\,}.
	\end{align*}
	Recalling that \Cref{obs:close} tells us that $\abs{f_\red(x,y) - f_\blue(x,y)} \leq 1$ for all $x,y\in V$, the claimed inequality follows.
\end{proof}

We remark here that inequality in \Cref{claim:step3} holds even if we condition on the value of $v$, and the proof runs exactly as above. 
However, the weaker statement both aligns with our notation $\EE_k$ and suffices for our needs.

We may combine Claims \ref{claim:step1}, \ref{claim:step2}, and \ref{claim:step3} to arrive at the following important corollary.

\begin{corollary}
	\label{claim:flip}
	The following holds for all $k\in[n-1]$.
	\begin{equation*}
		\EE_{k} \b[\big]{f(u,v)} \leq \EE_{k-1} \b[\big]{f(u,v)} + \EE_k\b[\big]{S^u(v)}.
	\end{equation*}
\end{corollary}

To complete the proof of \Cref{lem:key}, it is clear from \Cref{claim:flip} that it suffices to show the following.

\begin{claim}
	\label{claim:compute}
	Fix an integer $k\in [n]$. We have
	\begin{equation*}
		\EE_k \b[\big]{S^u(v)} \leq \frac{1}{2}\sqrt{\frac{k(n-k)}{n}}\p[\Big]{\frac{1}{n-k} + \frac{1}{k}}(1+o(1)).
	\end{equation*}
\end{claim}

\begin{proof}
	This claim in fact holds even if we condition on $v$. 
	Indeed, fix a vertex $x$ and assume that $x$ is in $r$ red edges.
	Let $u$ be conditioned to be at distance $k$ from $x$, and let $X$ be the random variable of how many of the red edges incident to $x$ point towards $u$ (that is, their other end is closer to $u$ than $x$ is).
	We can recall the definition \eqref{eq:s-def} of the set $S^u(x)$ and compute
	\begin{equation*}
		S^u(x) = \abs[\bigg]{\frac{\abs{J^u_\red(x)}}{n - k} - \frac{\abs{I^u_\red(x)}}{k}} 
		= \abs[\bigg]{\frac{r - X}{n - k} - \frac{X}{k}}.
	\end{equation*}
	We may note that the probability that $X$ is equal to some number $i$ is simply the probability that independently sampling $k$ of the edges incident to $x$ without replacement results in $i$ red edges, and so
	\begin{equation*}
		\prob{X = i} = \frac{\binom{r}{i}\binom{n - r}{k - i}}{\binom{n}{k}}.
	\end{equation*}
	In particular, $X$ has a hypergeometric distribution, for which the expectation and variance are known to be
	\begin{equation}
		\label{eq:true-facts-about-the-x}
		\EE[X] = \frac{k r}{n} \quad \text{ and } \quad \var{X} = \frac{k(n - r)(n - k)r}{n(n-1)(n-2)}.
	\end{equation}
	Indeed, write $X = Y + kr/n$, so that $\EE[Y] = 0$.
	We can easily compute then that
	\begin{equation*}
		S^u(x) = \abs{Y}\p[\Big]{\frac{1}{n-k} + \frac{1}{k}}.
	\end{equation*}
	But we may then apply Jensen's inequality to find that
	\begin{equation*}
		\EE\b[\big]{S^u(x)} = \p[\Big]{\frac{1}{n-k} + \frac{1}{k}} \EE\b[\big]{\,\abs{Y}\,} \leq \p[\Big]{\frac{1}{n-k} + \frac{1}{k}} \sqrt{\var{Y}}.
	\end{equation*}
	But $\var{Y} = \var{X}$, and so by \eqref{eq:true-facts-about-the-x}, we find that
	\begin{align*}
		\EE\b[\big]{S^u(x)} &\leq \p[\Big]{\frac{1}{n-k} + \frac{1}{k}} \sqrt{\frac{k(n - r)(n - k)r}{n(n-1)(n-2)}} \\
		&\leq \sqrt{\frac{k(n-k)}{n}}\p[\Big]{\frac{1}{n-k} + \frac{1}{k}}\sqrt{\frac{(n - r)r}{(n-1)(n-2)}} \\
		&\leq \frac{1}{2}\sqrt{\frac{k(n-k)}{n}}\p[\Big]{\frac{1}{n-k} + \frac{1}{k}}(1+o(1)),
	\end{align*}
	where the $o(1)$ term depends only on $n$, as required.	
\end{proof}

With \Cref{claim:compute} in hand, \Cref{lem:key} follows immediately by induction on $k$.
\qed


\section{Concluding remarks and open problems}
\label{sec:conc}

While our result does make progress towards \Cref{conj:path}, there is still a large gap between our bound of $O(\sqrt{n})$ and the conjectured constant. 
We note that improving the bound further will require some new ideas, as there are colourings of $E(Q_n)$ for which the expected minimum number of colour changes on a path from a vertex $v$ to its antipode $\bar{v}$ is bounded below by $\Omega(\sqrt{n})$.

Indeed, colour the edges of $Q_n$ in layers: pick an arbitrary vertex $x\in V(Q_n)$, and colour the edge $xy$ red if it is an even distance from $x$, and blue if it is at an odd distance.
Let $v\in V(Q_n)$ be selected uniformly at random, and let $L = d(x, v)$ be the random variable equal to the layer of $v$.
Then $L$ has a binomial distribution with parameter $p=1/2$, and so (omitting floors and ceilings for convenience) it can be computed that
\begin{equation*}
	\EE\b[\big]{\,\abs{L - n/2}\,} = 2^{-n}\sum_{j=0}^{n/2} \binom{n}{j}(n-2j) = 2^{-n}n\binom{n-1}{n/2} = \frac{\sqrt{n}}{2}(1 + o(1)).
\end{equation*}
and so on average a path needs roughly $\sqrt{n}$ colour changes to get to the antipodal vertex.
In combination with \Cref{thm:main}, this shows that the answer to \Cref{q:random} is $\Theta(\sqrt{n})$.
It seems reasonable to suspect, but potentially difficult to prove, that colouring by layers maximises the quantity in \eqref{eq:q-random}.

Of course, when colouring by layers, any vertex from the middle layer(s) has a monochromatic path to its antipodal vertex, which is also in the middle layer(s).
However, only a small proportion of vertices of $Q_n$ have this property, and detecting this in general would seem to require ideas beyond those presented here.

Beyond just proving a stronger bound, an advantage of the techniques presented here is the possibility of generalising to more colours.
In the interest of keeping the proofs here as concise as possible, we have not pursued this direction, though it seems that they would generalise to this setting.
We thus suggest a possible generalisation of \Cref{conj:path} to the following.

\begin{question}
	\label{q:general}
	If $E(Q_n)$ is $r$-coloured for some fixed $r \leq n$, must there be a pair of antipodal vertices of $Q_n$ and a geodesic path between them which changes colour at most $r-1$ times?
\end{question}

By colouring edges of $Q_n$ according to their direction, it is clear that \Cref{q:general}, if true, would be the optimal result.

Finally, we can also consider other structures in $Q_n$.
Indeed, it is natural to consider not just a path between two antipodal vertices, but one connecting all vertices: a Hamilton path.

\begin{question}
	\label{q:paths}
	Given $n$, what is the minimal number $h(n)$ such that, however $E(Q_n)$ is 2-coloured, there is a Hamilton path of $Q_n$ which changes colour at most $h(n)$ times? 
\end{question}

We outline an argument to show that there is a colouring of $E(Q_n)$ in which every Hamilton path has at least $\Theta(2^n/n)$ colour changes.
In fact, we show that any spanning tree of $Q_n$ has at least $\Theta(2^n/n)$ monochromatic components in this colouring.

If $n = 2^k - 1$, take $C\sseq Q_n$ to be a Hamming code (every vertex of $Q_n$ is either in $C$ or adjacent to exactly one element of $C$, and all elements of $C$ are pairwise at distance at least 3), and colour edges red if they are incident to $C$ and blue otherwise, then all vertices of $C$ are in different monochromatic components.
As $\abs{C} = 2^n / (n+1)$, we thus know that any spanning tree of $Q_n$ has at least $\Omega(2^n/n)$ monochromatic components.
Taking a suitable sub-cube of $Q_n$ allows this result to be extended to values of $n$ not of the form $2^k - 1$.

For general spanning trees, $\Theta(2^n/n)$ monochromatic components is the truth: a result of \cite{DDGZ01} shows that there is a spanning tree $T$ of $Q_n$ with at least $(1 - C/n)2^n$ leaves for some constant $C$.
Every leaf of $T$ must share a monochromatic component with its unique neighbour vertex, and every non-leaf vertex of $T$ is in at most two monochromatic components, and so $T$ has at most $O(2^n/n)$ monochromatic components.

We close by remarking that it does not seem at all obvious whether one should expect the $\Omega(2^n/n)$ lower bound for $h(n)$ as in \Cref{q:paths} to be the truth, or indeed if the value of $h(n)$ should even be $o(2^n)$.


\subsection{Acknowledgements}
The author would like to thank Bernando Subercaseaux for helpful comments on an earlier draft of the manuscript.
Thanks are also due to James Sarkies for suggesting the arguments presented at the end of \Cref{sec:conc}.


\bibliographystyle{abbrvnat}  
\renewcommand{\bibname}{Bibliography}
\bibliography{main}

\begin{thebibliography}{12}
\providecommand{\natexlab}[1]{#1}
\providecommand{\url}[1]{\texttt{#1}}
\expandafter\ifx\csname urlstyle\endcsname\relax
  \providecommand{\doi}[1]{doi: #1}\else
  \providecommand{\doi}{doi: \begingroup \urlstyle{rm}\Url}\fi

\bibitem[Alon et~al.(2006)Alon, Radoi{\v{c}}i{\'c}, Sudakov, and Vondr{\'a}k]{ARSV06}
N.~Alon, R.~Radoi{\v{c}}i{\'c}, B.~Sudakov, and J.~Vondr{\'a}k.
\newblock A {R}amsey-type result for the hypercube.
\newblock \emph{Journal of Graph Theory}, 53\penalty0 (3):\penalty0 196--208, 2006.

\bibitem[Baber et~al.(2023)Baber, Behague, Calbet, Ellis, Erde, Gray, Ivan, Janzer, Johnson, Mili{\'c}evi{\'c}, Talbot, Tan, and Wickes]{Imre23}
R.~Baber, N.~Behague, A.~Calbet, D.~Ellis, J.~Erde, R.~Gray, M.-R. Ivan, B.~Janzer, R.~Johnson, L.~Mili{\'c}evi{\'c}, J.~Talbot, T.~S. Tan, and B.~Wickes.
\newblock A collection of open problems in celebration of {Imre Leader's} 60th birthday.
\newblock \emph{arXiv preprint arXiv:2310.18163}, 2023.

\bibitem[Duckworth et~al.(2001)Duckworth, Dunne, Gibbons, and Zito]{DDGZ01}
W.~Duckworth, P.~E. Dunne, A.~M. Gibbons, and M.~Zito.
\newblock Leafy spanning trees in hypercubes.
\newblock \emph{Applied Mathematics Letters}, 14\penalty0 (7):\penalty0 801--804, 2001.

\bibitem[Dvo{\v{r}}{\'a}k(2020)]{Dvo19}
V.~Dvo{\v{r}}{\'a}k.
\newblock {A note on Norine's antipodal-colouring conjecture}.
\newblock \emph{The Electronic Journal of Combinatorics}, 27, article P2.26, 2020.

\bibitem[Feder and Subi(2013)]{FS13}
T.~Feder and C.~Subi.
\newblock On hypercube labellings and antipodal monochromatic paths.
\newblock \emph{Discrete Applied Mathematics}, 161\penalty0 (10-11):\penalty0 1421--1426, 2013.

\bibitem[Frankston and Scheinerman(2024)]{FS24}
K.~Frankston and D.~Scheinerman.
\newblock {Proving Norine's conjecture holds for $ n= 7$ via SAT solvers}.
\newblock \emph{arXiv preprint arXiv:2408.02474}, 2024.

\bibitem[Kirchweger et~al.(2025)Kirchweger, Peitl, Subercaseaux, and Szeider]{KPSS25}
M.~Kirchweger, T.~Peitl, B.~Subercaseaux, and S.~Szeider.
\newblock From the finite to the infinite: sharper asymptotic bounds on {N}orin's conjecture via {SAT}.
\newblock \emph{arXiv preprint arXiv:2511.08386}, 2025.

\bibitem[Leader and Long(2014)]{LL14}
I.~Leader and E.~Long.
\newblock Long geodesics in subgraphs of the cube.
\newblock \emph{Discrete Mathematics}, 326:\penalty0 29--33, 2014.

\bibitem[Long(2013)]{Lon13}
E.~Long.
\newblock Long paths and cycles in subgraphs of the cube.
\newblock \emph{Combinatorica}, 33\penalty0 (4):\penalty0 395--428, 2013.

\bibitem[Norine(2008)]{Nor08}
S.~Norine.
\newblock Edge-antipodal colorings of cubes.
\newblock Open Problem Garden. \url{http://www.openproblemgarden.org/op/edge_antipodal_colorings_of_cubes}, 2008.
\newblock Accessed: 2026-05-08.

\bibitem[West and Wise(2018)]{WW18}
D.~B. West and J.~I. Wise.
\newblock Antipodal edge-colorings of hypercubes.
\newblock \emph{Discussiones Mathematicae Graph Theory}, 39\penalty0 (1):\penalty0 271--284, 2018.

\bibitem[Zulkoski et~al.(2017)Zulkoski, Bright, Heinle, Kotsireas, Czarnecki, and Ganesh]{ZBHKC17}
E.~Zulkoski, C.~Bright, A.~Heinle, I.~Kotsireas, K.~Czarnecki, and V.~Ganesh.
\newblock {Combining SAT solvers with computer algebra systems to verify combinatorial conjectures}.
\newblock \emph{Journal of Automated Reasoning}, 58\penalty0 (3):\penalty0 313--339, 2017.

\end{thebibliography}


\end{document}